\documentclass[12pt]{amsart}
\usepackage[margin=1in,marginparwidth=0.8in, marginparsep=0.1in]{geometry}

\pdfoutput=1

\usepackage{graphicx}

\usepackage{color}

\definecolor{darkgreen}{rgb}{0.0, 0.7, 0.0}
\definecolor{purple}{rgb}{0.5, 0.0, 0.5}

\usepackage[draft]{say}
\newcommand{\sayVS}[1]{\say[VS]{\color{blue}{\bf VS:}\;#1}}
\newcommand{\sayLM}[1]{\say[LM]{\color{darkgreen}{\bf LM:}\;#1}}

\usepackage{amsfonts,amssymb,latexsym,amsmath}
\usepackage{times}
\usepackage{mathrsfs}

\usepackage{amsfonts}
\usepackage[english]{babel}
\usepackage[leqno]{amsmath}
\usepackage{amssymb,amsthm}
\usepackage{amscd}
\usepackage{enumerate}
\usepackage{epsfig}

\usepackage[matrix,arrow,tips,curve]{xy}

\input{xy}
\xyoption{all}

\renewcommand{\AA}{\mathbb{A}}

\newcommand{\CC}{\mathbb{C}}
\newcommand{\DD}{\mathbb{D}}

\newcommand{\FF}{\mathbb{F}}

\newcommand{\HH}{\mathbb{H}}

\newcommand{\PP}{\mathbb{P}}
\newcommand{\QQ}{\mathbb{Q}}
\newcommand{\RR}{\mathbb{R}}

\newcommand{\ZZ}{\mathbb{Z}}

\newcommand{\cC}{\mathcal{C}}
\newcommand{\fF}{\mathcal{F}}

\newcommand{\hH}{\mathcal{H}}
\newcommand{\lL}{\mathcal{L}}

\newcommand{\oO}{\mathcal{O}}

\theoremstyle{plain}

\newtheorem{lemma}{Lemma}[section]
\newtheorem{proposition}[lemma]{Proposition}
\newtheorem{corollary}[lemma]{Corollary}
\newtheorem{theorem}[lemma]{Theorem}

\theoremstyle{definition}
\newtheorem{definition}[lemma]{Definition}

\theoremstyle{remark}
\newtheorem{example}[lemma]{Example}

\newtheorem{remark}{Remark}

\newcommand{\Z}{\mathbb{Z}}
\newcommand{\Q}{\mathbb{Q}}

\newcommand{\ov}{\overline}

\def \PP{\mathbb{P}}

\def\M0{\mathcal M^0}

\def \M{\mathcal M}

\DeclareMathOperator{\codim}{{codim}}

\newcommand{\cF}{{\mathcal F}}
\newcommand{\cG}{{\mathcal G}}

\newcommand{\cS}{{\mathcal S}}

\newcommand{\hra}{\hookrightarrow}

\title{Higher discriminants and the topology of algebraic maps}
\author{Luca Migliorini}
\author{Vivek Shende}

\begin{document}

\begin{abstract}
We show that the way in which Betti cohomology varies in a proper family of
complex algebraic varieties is controlled by certain `higher discriminants' in the base.  
These discriminants are defined in terms of transversality conditions, which in the case 
of a morphism between smooth varieties can be checked by a tangent space calculation.  
They control the variation of cohomology in the following two senses: 
(1) the support of any summand of
the pushforward of the $\mathrm{IC}$ sheaf along a projective map is a component of a higher discriminant,
and (2) any component of the characteristic cycle of the proper pushforward of the constant function is 
a conormal variety to a component of a higher discriminant.  

The same would hold for the Whitney stratification of the family, but there are vastly fewer
higher discriminants than Whitney strata.  For example, in the case of the Hitchin fibration, 
 the stratification by higher discriminants gives exactly the $\delta$ stratification introduced
by Ng\^o. 
\end{abstract}

\maketitle

\thispagestyle{empty}

\section{Introduction}

Let $f:Y  \to X$ be a proper family of algebraic varieties.    
Suppose as given a satisfactory understanding of the cohomology of the 
general fibre, and perhaps 
of some particularly well behaved special fibres.  What can be said about 
the cohomology of an arbitrary fibre?


Where the map $f$ is smooth, all fibres have the same cohomology.
%
Over $\CC$, the theory of Whitney stratifications assures us of the existence of some stratification of $X$
so that topological properties of the fiber are constant in each stratum.  Any question
about the cohomology of fibres could be answered by checking at the generic point of every
Whitney stratum.  Unfortunately there tends to be a vast number of strata -- at least as many as 
topological types of fibres -- which are moreover hard to characterize even in small examples.  
Checking any fact on each stratum is prohibitive.  
 
Our purpose here is to explain a better strategy.  First we review two points of view on
how to think of cohomology in families. For simplicity in this introduction we focus on {\em complex} varieties. 
\vspace{2mm}

\subsection{Microlocal geometry}\label{mg}
MacPherson's Chern class morphism from constructible functions to homology,  
$$\mathbf{c} : Con(X) \to H_*(X)$$
commutes with proper pushforwards, and is normalized by the condition
that for smooth, proper $X$ we have $\mathbf{c}({\mathbf 1}_X) = c(TX) \cap [X]$. 
Any answer to our question almost necessarily involves coming to terms with it: 
we want to understand the family of cohomologies
of fibres $Rf_* \QQ$; a simpler question is to understand the family of Euler characteristics of fibres $f_* {\mathbf 1}_X$,
and a still simpler question is to understand 
$\mathbf{c}(f_* {\mathbf 1}_X) = f_* \mathbf{c}({\mathbf 1}_X)$. 


We recall the microlocal perspective on these issues \cite{KS, Ken}.  To a constructible
function on a smooth variety $Z$, one can assign the conical Lagrangian of 
co-directions in $T^*Z$ along which it fails to remain constant; this is called the {\em singular support}.
Each component
can be weighted by the generic amount by which the function changes; 
the result is a conical Lagrangian cycle in 
$T^*Z$ called the {\em characteristic cycle}. 
This gives an isomorphism between the free abelian groups of constructible functions and conical
Lagrangian cycles.  The space of Lagrangian cycles has a geometrically natural basis: 
closures of conormals to smooth, locally closed subvarieties.  So we expand

\begin{equation} \label{eq:cc} CC(f_* {\mathbf 1}_Y) = \sum_\alpha n_\alpha [\overline{T_{V_\alpha}^* X}] \end{equation}

To return to constructible functions, we invoke the 
 Brylinski-Dubson-Kashiwara local index formula \cite{BDK, G}.  
\begin{eqnarray*}
CC^{-1}: Lag(T^* Z) & \xrightarrow{\sim} & Con(Z) \\
\big[ \overline{T_S^* Z} \big] & \mapsto & (-1)^{\codim S} Eu_S 
\end{eqnarray*}

The ``Euler obstruction'' $Eu_S$ is a constructible function intrinsic to a variety $S$
and taking value $1$ at smooth points, originally defined by 
MacPherson in terms of the Nash transform.  By the above index formula, we
have 
$$f_* 1_Y = \sum_\alpha (-1)^{\codim V_\alpha} n_\alpha \cdot Eu_{V_\alpha}$$

There is a pushforward of conical
Lagrangian cycles, commuting with the characteristic cycle transformation \cite{M, Ken, KS}.  
From $f: Y \to X$, one forms its $\tilde{f}: T^* X \times_X Y \to T^* X$ and its differential
$df: T^* X \times_X Y \to T^* Y$.  These give a correspondence
$$T^*Y \xleftarrow{df} T^* X \times_X Y  \xrightarrow{\tilde{f}} T^*X$$
This correspondence gives the pushforward of Lagrangian cycles, and in particular: 
$$SS(f_* {\mathbf 1}) \subseteq \tilde{f}(  df^{-1} (0_Y))$$ 

Using this formula in practice means figuring out what the components $V_\alpha$ of 
$ \tilde{f}(  df^{-1} (0_Y))$ are.  The purpose of the present work is to give a characterization
of these components in terms of more readily computable quantities.

\subsection{The decomposition theorem}\label{dt}

The second point of view comes from the decomposition theorem of Beilinson, Bernstein, and Deligne \cite{BBD}. 
A special case of this asserts that if $Y$ is smooth and $f:Y \to X$ is projective, then 
\begin{equation} \label{eq:decomposition} Rf_* \QQ = \bigoplus_\alpha IC(U_\alpha, \mathcal{L}_\alpha)[n_\alpha] \end{equation}
In words, $Rf_* \QQ$ splits as a direct sum of shifted simple perverse sheaves. 
One such simple perverse sheaf, generically on its support $U_\alpha$, is 
just a simple locally constant sheaf $\mathcal{L}_\alpha$; the perverse sheaf can be 
uniquely recovered from $\mathcal{L}_\alpha$ by a complicated procedure called intermediate extension.  
This separates the problem into two steps: first, determine the $U_\alpha$ and $\mathcal{L}_\alpha$; 
second, understand the intermediate extension.  

The decomposition theorem itself leaves the identity of the $U_\alpha$ as a mystery, but 
some results constraining them are known. 
If $Y$ is nonsingular and the map $f$ is semismall, i.e. 
$\codim \{ \dim Y_x = i \} \ge 2i$, then the $U_\alpha$ which appear are precisely the closures of the components of the 
loci $\{ \dim Y_x = i \}$ for which the above inequality is an equality \cite{BM,dCM}.  For any $f$,  
if the maximum fibre dimension is $d$, then the $U_\alpha$ which appear can be shown by the relative
Hard Lefschetz theorem to have codimension $\le d$, with strict inequality over the locus where 
the fibres are irreducible (``Goresky-MacPherson inequality", see \cite{N}, Th\'eor\`eme 7.3.1).  
Finally, there is the recent and celebrated result of Ng\^o \cite{N, N1}, which applies to
the case when $f:Y\to X$ is a certain sort of relative compactification of a sufficiently well behaved abelian scheme,
and implies for instance that all the $U_\alpha$ which occur are contained in the locus where the fibres $Y_x$ 
are not integral.

\subsection{Higher discriminants}

We now rephrase our goal: to understand $f_* {\mathbf 1}_Y$ and $Rf_* \QQ_Y$, we want to identify the varieties 
$V_\alpha$ which appear in Equation \ref{eq:cc} and the varieties $U_\alpha$ which appear in 
Equation \ref{eq:decomposition}.  To this end, we introduce the {\em higher discriminants}.  When $X$ is smooth 
 (a more general formulation appears in Sec. \ref{sec:sing}), these are given by: 

$$\Delta^i(f):= \{ x \in X \, |\, \mbox{no
$(i-1)$-dimensional subspace of $T_x X$ is transverse to $f$}\,\}$$

This determines a stratification

$$X = \Delta^0(f) \supset \Delta^1(f) \supset \Delta^2(f) \supset \Delta^3(f) \supset \cdots$$

Observe that $\Delta^1(f)$ is by definition the locus where the fibre is singular -- that is, the 
usual discriminant.  By generic smoothness, \begin{equation}\label{codest}
\codim \Delta^i(f) \ge i.
\end{equation}

We think about these discriminants in the following way.  
Moving $\delta \in \Delta(f)$ off the discriminant to  
$\not  \!\! \delta \notin \Delta^1(f)$ 
changes the fibre topology: $Y_\delta \not \sim Y_{\not\, \delta}$.  
{\em But we can blur our focal point to obscure this feature}: we pass
to a one dimensional disc  $\DD \ni \delta$, chosen generic and small enough to retract $f^{-1}(\DD)=:Y_\DD \sim Y_{\delta}$. 
A one dimensional disc cannot be perturbed off the discriminant, and indeed for $\delta$ general 
in $\Delta^1(f)$, a perturbation $\DD'$ of the thickening $\DD$ induces a homeomorphism $Y_{\DD'} \sim Y_{\DD}$.  
The higher discriminant $\Delta^2(f)$ is the locus which still appears to our blurred vision: 
where even a general perturbation of a general one parameter thickening changes the fibre topology.

Our main result is that all the $V_\alpha$ and $U_\alpha$ of \S \ref{mg} and \ref{dt} are irreducible components of higher discriminants. 

\vspace{3mm} {\bf Theorem A.} 
If $f:Y \to X$ is a proper map of algebraic varieties, then
any component of the characteristic cycle of $f_*{\mathbf 1}_Y$ is the conormal variety to an $i$-codimensional
component of $\Delta^i(f)$ for some $i$.
\vspace{3mm}

\vspace{3mm} {\bf Theorem B.}
If $f: Y \to X$ is a projective map of algebraic varieties,  then
any component of the support of a summand of $Rf_* IC_Y$ is an $i$-codimensional component
of $\Delta^i(f)$ for some $i$.
\vspace{3mm}

We will deduce Theorem A from:

\vspace{3mm} {\bf Theorem C.}
If $f: Y \to X$ is a proper map of smooth algebraic varieties, $\tilde{f}(  df^{-1} (0_Y))$ is the union 
over all $i$ of
the conormals to all the $i$-codimensional components of $\Delta^i(f)$. 
\vspace{3mm}

In fact, in characteristic zero, this implies Theorem B as well: the conormal to the 
support of any summand of a sheaf is necessarily a component of the singular support
(note here it is essential we use singular support rather than characteristic cycle, 
where cancellation may arise), so any such will be a component of $SS(f_* \Q) \subseteq \tilde{f}(  df^{-1} (0_Y))$. 

In characteristic $p$, however, we do not have a good notion of 
constructible functions or characteristic cycles.  The problem is wild ramification: 
for instance, the map $$as: \AA^1 \xrightarrow{x \mapsto x^p - x} \AA^1$$ is
a nontrivial cover whose source is a space with Euler characteristic one; 
composing with the pushforward to a point gives a counterexample to the possibility
that the 
naive pushforward of constructible functions is functorial.  No 
general solution to this problem is currently known. 

Theorem B, however, can at least be sensibly stated in any characteristic.  We give a
proof in \S \ref{sec:support} which works generally, under additional assumptions on the codimension of
the higher discriminants, which are automatic in characteristic zero. 

\subsection{Applications}

Theorems A, B, C, factor the problem of understanding $f_* {\mathbf 1}$ or $Rf_* \Q$ 
into two pieces: first, by computing derivatives, determine the loci $\Delta^i(f)$; second,
understand their singularities well enough to compute Euler obstructions or intersection cohomology 
sheaves.  
Note the second
step
{\em no longer depends on the function $f$}.  

We remark briefly on two ways that such results may be used.  One way, following \cite{N}, 
is that given two maps $f: X \to B$ and $g: Y \to B$, one can deduce comparisons of cohomology
on all fibers from comparisons of cohomology on sufficiently generic fibers.  Our results here imply
that ``sufficiently generic'' means ``the general point of each higher discriminant''.  We have used this 
method in \cite{MS, MSV}; or more precisely, our definition here of higher discriminants is distilled
from the method we used in \cite{MS}, and our results here are applied in \cite{MSV}.  

Of course, 
doing so requires computing the higher discriminants.  This, however, is a tangent space calculation. 
For example, in Section \ref{sec:dima}, we describe (following ideas of Dima Arinkin) how the higher
discriminants can be computed in the case of an integrable system.  This recovers the support 
theorem of \cite{N} in this case, but in fact more: it characterizes the microsupport as well.

A slightly different application appears in \cite{ST}.  Here, we are interested in bounding the cohomology 
of fibers of a given map $f:X \to B$ in terms of the cohomology of the general fibre and the singularities 
of the higher discriminants.  The general method of doing this is described here in Section \ref{sec:sab}, 
in part following \cite{Ma}.  

\vspace{2mm} 
\noindent {\bf Acknowledgements:}  
We thank Paolo Aluffi, Daniel Lowengrub,  Ng\^o Bau Chau, Filippo Viviani, Jacob Tsimerman, Melanie Wood, 
and Geordie Williamson for helpful discussions,
and especially Dima Arinkin for the characterization of higher discriminants of integrable systems, 
and Mark A. de Cataldo, who pointed out several inaccuracies.

\section{From singular support to higher discriminants}
In this section varieties are assumed to be real analytic manifolds. By a ``sheaf'' we mean a complex $\fF$ of sheaves of abelian groups, such that 
its cohomology sheaves $\hH^r(\fF)$ are constructible, and vanish for $|r|>>0$. We denote by
$\mathrm{D}^b_c(X)$ the corresponding derived category.
By \cite{V2}, every $\fF$ is constructible with respect to a Whitney stratification.

\subsection{Review of singular support of sheaves and transversality}
To a  complex of sheaves $\fF$ on a real manifold $X$, one can assign the locus of 
co-directions in the cotangent bundle along which it fails to remain locally constant.  This is 
its singular support, a conical Lagrangian subvariety of the cotangent bundle  \cite[Chap.V]{KS}, and \S 3 of \cite{VW}  for a clear recollection of the properties of the singular support. 
A covector $p=(x,\xi) \in T^*X$ does not belong to  $SS(\fF)$ 
if there exists an open neighborhood $U$ of  of $p$ such that for any $x' \in X$ and any real function $\phi$  defined in 
a neighborhood of $x'$ such that $d\phi_{x'} \in U$  and $\phi(x') = 0$, one has 
$$\left(R\Gamma_{\{\phi \geq 0 \}}\fF\right)_{x'}=0.$$
From its very definition $SS(\fF)$ is a conical subset, i.e. invariant by the action of $\RR_{>0}$ on $T^*X$. The definition
of singular support makes sense for an arbitrary bounded complex of sheaves, and by \cite[Theorem 6.5.4]{KS}, $SS(\fF)$ is an involutive (or co-isotropic) subset of $T^*X$. If $\fF$ is constructible then 
$SS(\fF)$ is furthermore Lagrangian (\cite[Theorem 8.4.2]{KS}).

Singular support transforms well under certain natural operations.  
Given a map $f: Y \to X$ of real analytic manifolds, 
we write 
$$T^* Y   \xleftarrow{df} Y \times_X T^* X   \xrightarrow{\tilde{f}}  T^* X $$
We write $f_\dagger :=  \tilde{f} \circ (df)^{-1}$ and $f^\dagger = (df) \circ \tilde{f}^{-1}$ for the corresponding maps between subsets of $T^*Y$ 
and subsets of $T^*X$.  

The convolution $f_\dagger$ controls singular supports of proper pushforwards: 
according to \cite[Prop. 5.4.4]{KS}, we have  
\begin{equation} \label{eq:shvs}
SS(Rf_* \cF) \subseteq f_\dagger SS(\cF)
\end{equation}
if the restriction of $f$ to the support of $\cF$ is proper. 



\begin{definition}\label{trsv}
Given a conical subset  $A \subseteq T^*X$, we say  $f: Y \to X$ is transverse to $A$ if 
$$A \cap f_\dagger (T_Y^* Y) \subseteq T_X^* X,$$
i.e. no nonzero covector $(x,\xi)$ in $A$ annihilates $f_* T_y Y$ for $y \in Y_x$.
We say a submanifold $U \subseteq X$ is transverse to $A$ if the inclusion map is transverse
to $A$.  Finally, we say that a subspace $V \subseteq T_x X$ is transverse to $A$ if no nonzero
covector in $A_x$ annihilates $V$.  
\end{definition}

\begin{remark}\label{rem:dim_triv}
If $A_x$ is a vector subspace of $T_xX^*$, hence of the form $A_x=U^\perp$, with $U \subseteq T_x X$, then no subspace $W \subseteq T_x X$ with 
$\dim W < \dim A_x$ is transverse to $A_x$, while there exist an open dense subset 
of $\dim A_x$-dimensional subspaces of $T_x X$ 
which are transverse to $A_x$. 
\end{remark}
%

\begin{lemma} (Transversality.)  Let $f: Y \to X$, and $g: Z \to X$.   
Then the following are equivalent: 
\begin{itemize}
\item The map $f$ is transverse to $g_\dagger (T_Z^* Z)$. 
\item The map $g$ is transverse to $f_\dagger (T_Y^* Y)$. 
\item At any point $(y, z)$ with $f(y) = g(z)$, we have $f_*(T_y Y) + g_*(T_z Z) = T_x X$. 
\end{itemize}
In this case, $Y \times_X Z$ is smooth.  (In the category of schemes, 
the converse holds as well.)  
\end{lemma}
The singular support behaves well with respect to transverse pullback \cite[Prop. 5.4.13]{KS}: 
if $\cF$ is a sheaf on $X$ and $f: Y \to X$ is transverse to $SS(\cF)$, then 
$$SS(f^* \cF) \subseteq f^\dagger SS(\cF).$$

\subsection{Review of constructible functions}
In this section we work in the complex analytic set-up.
The group of constructible functions ${\mathscr C}(Y)$ on a complex analytic 
variety $Y$ is the free abelian group generated by 
characteristic functions of closed analytic subvarieties.  We recall below some relevant facts; for a detailed treatment
see e.g. \cite[\S 9.7]{KS} or \cite[\S 2.3]{sc}.

A map
$f: X \to Y$ induces a proper pushforward $f_!: {\mathscr C}(X) \to {\mathscr C}(Y)$:
if $W \subseteq X$ is a closed subvariety and ${\mathbf 1}_W$ is its characteristic function, then:
$f_!{\mathbf 1}_W(y):= \chi_c(f^{-1}(y) \cap W)$.  We have written the compactly supported Euler
characteristic to emphasize its additivity, but recall that for algebraic varieties, the compactly supported and
usual Euler characteristics agree (see e.g. \cite[p. 141]{F} or \cite[\S 6.0.6]{sc}). 
We employ Viro's integral notation \cite{Vi}: if $\pi_Y: Y \to \mathrm{point}$ is the structure map, then for $\xi \in \mathscr{C}(Y)$ 
we write $\int_Y \xi d\chi := (\pi_Y)_! \xi$.  Explicitly, 
if $\xi=\sum \xi_{\alpha}{\mathbf 1}_{Y_{\alpha}}$, then 
\[
\int_{Y} \xi d \chi :=\sum_{\alpha}\xi_{\alpha}\chi_c(Y_{\alpha}),
\] 
Because all the strata are locally contractible, we have:

\begin{lemma} Fix $\xi \in \mathscr{C}(Y)$.  Then for any $y \in Y$, there exists $\epsilon_0>0$ such that, for $0 <\epsilon \leq \epsilon_0,$ 
\[
\xi(y)= \!\!\! \int\limits_{ \| z \|< \epsilon } \!\!\! \xi d \chi. 
\]
\end{lemma}

We recall that, for a subvariety $V$, one can define
the Euler obstruction $\mathrm{Eu}_V(x)$ (see \cite{M}),  a constructible function with support on $V$;
$\mathrm{Eu}_V(x)=1$ if $x$ is a smooth point.
It is constant along strata of a Whitney stratification and preserved by taking products with smooth spaces.  
There is a hyperplane formula: assuming $\dim V > 0$ 
and taking a local embedding at some $v \in V \subseteq \CC^n$, 
\begin{equation} \label{eq:eulob}
\int\limits_{\DD^{n-1}} \!\!\! \mathrm{Eu}_V d \chi = \mathrm{Eu}_V(v)
\end{equation}
for a general disc $\DD^{n-1}$ passing near (but not through) $v$ \cite{BDK} (see also
\cite[Thm. 3.1]{BLS} and \cite{sc1}). The functions $\mathrm{Eu}_V$ give a basis for ${\mathscr C}(X)$.

\bigskip
Constructible functions are also associated with constructible sheaves. 
If $\fF \in \mathrm{D}^b_c(X)$ we have the constructible function
\[
[\fF] \in {\mathscr C}(X), \qquad [\fF](x)= \sum (-1)^i \dim {\mathcal H}^i(\fF)_x,
\]
The map $\fF \to [\fF]$ 
factors through the Grothendieck group, and is compatible with pushforward: $f_* [\fF] = [\mathrm{R}f_* \fF]$. 
Clearly ${\mathbf 1}_W=[\QQ_W]$, for a closed subvariety $W$.

We will employ the formalism of nearby and vanishing cycles 
(see \cite[\S 8.6]{KS} for their definition and the convention for  shifts employed here); we recall here enough to fix notation. 
Let  $f:Y \to \CC$ be a regular
function and $X_0= f^{-1}(0)$.   Consider some $x \in X_0$, and fix sufficiently
small $\epsilon \gg |\delta| > 0$.  Then we have the long exact sequence in cohomology
for the $\epsilon$-ball relative to the Milnor fibre: 
$$\HH^* (B_\epsilon(x), B_\epsilon(x) \cap f^{-1}(\delta) ;  \cF) \to 
\HH^*(B_\epsilon(x); \cF) \to \HH^*(B_\epsilon(x) \cap f^{-1}(\delta); \cF) \xrightarrow{[1]}$$
These are in fact the stalks at $x$ of a distinguished triangle of complexes of sheaves on $Y_0$: 
\begin{equation}\label{phipsitri}
\Phi_f \cF \to \cF|_{X_0} \to \Psi_f \cF \xrightarrow{[1]}.
\end{equation}

The functors $\Phi, \Psi$ descend to operators on constructible functions \cite[Prop. 3.4, Prop. 4.1]{V2}.
Explicitly if $X$ is a complex space, $l: X \to \CC$ a holomorphic function,
we may write vanishing and nearby operators on constructible functions 
$\Psi_l,\Phi_l:{\mathscr C}(M) \to {\mathscr C}(l^{-1}(0))$,
setting, for a constructible function $\xi: X \to \Z$, a point $p \in l^{-1}(0)$, and $\epsilon \gg \delta$,
\begin{equation}\label{phiconstr} \Psi_l \xi (p)\,\,\, :=  \!\!\!\!\!\!\!  \int\limits_{B_\epsilon(p) \cap l^{-1}(\delta)} \!\!\!\!\!\!\! \xi d\chi \,\,\,\,\,\,\,\,\,\,\,\,\,\,\,\,\,\,\,\,\,\,\,\,\,\,\,\,\,\, \Phi_l \xi(p) := \xi(p) - \Psi_l \xi(p). 
\end{equation}
If $\fF$ is a constructible sheaf, evidently
$\Phi_l [\fF] = [\Phi_l \fF]$ and similarly $[\Psi_l \fF] = \Psi_l [\fF]$.  

We may rewrite the hyperplane formula (\ref{eq:eulob}) for the Euler obstruction as:
$$ (\Phi_l \mathrm{Eu}_V)(x) = 0 \,\,\,\,\, \mbox{for $\dim V > 0$ and $l: V \to \CC$ a general coordinate function near $x$}. $$

\begin{definition}
(Singular support of a constructible function)
For $\xi \in \mathscr{C}(X),$ a covector $p=(y,\lambda_0) \in T^*X$ does not belong to $SS(\xi)$ 
if there exists an open neighborhood $U$ of  of $p$ such that for any $x' \in X$ and any real function $l$  defined in 
a neighborhood of $x'$ such that $dl_{x'} \in U$  and $l(x') = 0$, one has $\Phi_l(\xi)=0$.
\end{definition}
Note that the definition is compatible with the map
$\mathrm{D}^b_c(X) \to {\mathscr C}(X),$ in the sense that we have the inclusion
$SS([\cF]) \subseteq SS(\cF)$, which may be proper as the example of $\QQ_0\oplus \QQ_0[1]$ on $\AA^1$ shows.
If $\xi$ is constructible with respect to a decomposition $\mathcal S=\{S_\alpha \}$ of $X$, which, 
after refining we may suppose to be a Whitney stratification, then
\begin{equation}\label{eq:constr_fun}
SS(\xi) \subseteq \bigcup T^*_{S_\alpha}X.
\end{equation}

It follows from (\ref{eq:shvs}) that
\begin{equation} \label{eq:con}
SS(f_* {\mathbf 1}) \subseteq f_\dagger (T_Y^* Y) = 
\{(x, \varphi) \in T^* X \, | \, \exists y \in Y_x\,\, \mathrm{s.t.}\, \, \varphi (f_* T_y Y) = 0\}
\end{equation}

\subsection{Higher discriminants}

In order to organize the discussion of transversality, we introduce the {\em higher discriminants}: 

\begin{definition}
Let $C \subseteq T^* X$ be a conical subvariety.  We define
$$\Delta^i(C) := \{x \in X\, | \, \mbox{no $(i-1)$-dimensional subspace of $T_x X$ is transverse to $C$}\}.$$
Clearly, when we work in the complex analytic category, we refer to complex dimension in the definition above.

We write $\Delta^i(C)_{\mathrm{reg}}$ for the locus where $\Delta^i(C)$ is locally a manifold of codimension $i$. 
\end{definition}
\begin{remark}
If $C' \subseteq C$, then $\Delta^i(C') \subseteq \Delta^i(C)$ and $\Delta^i(C')_{\mathrm{reg}} \subseteq \Delta^i(C)_{\mathrm{reg}}$. 
\end{remark}

As we discussed above, by \cite{KS}, the singular support of a {\em constructible}
sheaf or function is conical Lagrangian. If $\fF$ is constructible with respect to the Whitney stratification $\mathcal S=\{S_\alpha \}$ of $X$, 
then 
$$
SS(\fF) \subseteq \bigcup T^*_{S_\alpha}X,
$$
where, given a locally closed submanifold $Z\subseteq X$,  $T^*_{Z}X$ stands for its conormal bundle. Similarly for a function as stated in (\ref{eq:constr_fun}).

\bigskip
It turns out that the higher discriminants control the 
decomposition of a conical Lagrangian into irreducible components.  To state the result precisely,
recall that in the category of real manifolds, conical only means ``invariant under $\RR_{\ge 0}$''; 
we write $\RR C$ to take the negative scalars as well.  This distinction will
be forgotten when we return to complex geometry. 
\begin{theorem} \label{thm:disclag}
If $C \subseteq T^*X$ is a closed conical subanalytic Lagrangian subset, then $\codim \Delta^i(C) \ge i$. 
Moreover, 
$$\RR C = \bigcup_i \overline{T_{\Delta^i(C)_{\mathrm{reg}}}^* X}$$
\end{theorem}
\begin{proof}
Choose a Whitney subanalytic stratification of $X$ so that $\RR C$ is contained in a union of 
closures of conormals to strata (see \cite[Proposition 8.3.10]{KS}).  
Observe that the Whitney condition A amounts to the assertion that if
$S'$ is a stratum inside the closure of $S$, then $\overline{T^*_S X}|_{S'} \subseteq T^*_{S'} X$. 
It follows that $C$ is contained in the union of (not closures of) conormals to strata; increasing
$C$ to this union only enlarges $\Delta^i(C)$ so we are free to do this. 

To now prove the first claim, we need only consider strata of codimension $\le i$, above which $C$
is the conormal, to which any general $i$ dimensional subspace of the tangent space is
transverse. 

As for the ``moreover'', the containment $T_{\Delta^i(C)_{\mathrm{reg}}}^* X \subseteq \RR C$ is obvious.  
For the reverse inclusion, observe that if 
$C \ne \bigcup_i \overline{T_{\Delta^i(C)_{\mathrm{reg}}}^* X}$, then there must be a full dimensional
component $C_0$ of $C$ not included in this union.  At a general point $c_0 \in C_0$, 
we can identify $C_0 = T^*_V X$ for some $V$.  But then $V \subseteq \Delta^{\codim V} (C)$; 
since $\codim \Delta^{\codim V} (C) = \codim V$, we have equality along the smooth locus.  
This is a contradiction. 
\end{proof}

We are interested in applying this result to a conical Lagrangian of the form $f_\dagger (T_Y^* Y)$.
We have immediately (see Remark \ref{rem:dim_triv}):

\begin{lemma} \label{lem:discfun}
Let $f:Y \to X$ be a map of smooth manifolds.  Then $\Delta^i(f_\dagger T_Y^* Y)$ is 
the locus in $X$ of points $x \in X$ such that no $(i-1)$-dimensional subspace of $T_x X$ 
is transverse to $f$. 
\end{lemma}

\begin{definition}(Higher discriminants of a map)
We write $\Delta^i(f) := \Delta^i(f_\dagger T_Y^* Y)$.
\end{definition}

\begin{remark}
If $f$ is proper, then $\Delta^i(f)$ is closed and moreover $\Delta^{>i}(f) \subseteq \Delta^i(f)$, by
Sard's theorem.  
\end{remark}

\begin{definition}(Higher discriminants of constructible sheaves and functions)
If $F$ is a constructible function or a complex of constructible sheaves on $X$, we set 
$\Delta^i(F) := \Delta^i(SS(F))$.  
\end{definition}

Translating the definitions of the higher discriminants
and the singular support, we have: 

\begin{lemma}
$\Delta^i(F)$ is the locus of $x \in X$ such that for the general $i-$dimensional local complete intersection $\DD^i \ni x$ and general 
linear form $l: (\DD^i,x) \to (\CC,0)$ 
we have $\Phi_l(F|_{\DD^{\,i}})(x) \ne 0$. 
\end{lemma}

Theorems A,B,C, whose statements we recall, follow immediately from the discussion above:

\vspace{3mm} {\bf Theorem C.}
If $f: Y \to X$ is a proper map of smooth algebraic varieties, $f_\dagger (T_Y^* Y)=\tilde{f}(  df^{-1} (T^*_Y Y))$ is the union 
over all $i$ of
the conormals to all the $i$-codimensional components of $\Delta^i(f)$. 
\vspace{3mm}
\begin{proof}
Follows from Theorem \ref{thm:disclag} and Lemma \ref{lem:discfun}. 
\end{proof} 

\vspace{3mm} {\bf Theorem A.} 
If $f:Y \to X$ is a proper map of algebraic varieties, then
any component of the characteristic cycle of $f_*{\mathbf 1}_Y$ is the conormal variety to an $i$-codimensional
component of $\Delta^i(f)$ for some $i$.
\begin{proof}
Follows from Theorem C by Equation \ref{eq:con}, i.e., 
\cite[Prop. 5.4.4]{KS}. \end{proof} 

\vspace{3mm} {\bf Theorem B.}
If $f: Y \to X$ is a projective map of algebraic varieties,  then
any component of the support of a summand of $Rf_* IC_Y$ is an $i$-codimensional component
of $\Delta^i(f)$ for some $i$.
\vspace{3mm}
\begin{proof}
The statement follows by applying the decomposition theorem and 
then the observation that if
$\cF$ is a summand of $Rf_* \Q_Y$, then the conormal to the support of $\cF$ is
a component of $SS(Rf_* \Q_Y)$. 
\end{proof}

We end this section with two examples from complex geometry clarifying the notion of higher discriminant:
\begin{example}(Higher discriminants and stratification by  rank of the differential)
The locus $\Delta^i(f)$
{\em may be larger} than the image of the locus in
$Y$ where $df$ has cokernel of dimension at least $i$:
Let $f$ be the blow-up of a point $p$ on a smooth surface: the  image of $df$ has dimension one at every point of the exceptional divisor.
Nevertheless $p \in \Delta^2(f)$: no one-dimensional disc in the base is transverse to all these tangent space images simultaneously, and indeed
the inverse image of every one-dimensional disc $\DD$ through $p$ will be given by the union of 
the proper transform of $\DD$ and the exceptional divisor, which is singular.
\end{example}
On the other hand it may well happen that a stratum for the map is not a higher discriminant:
\begin{example}({The versal deformation of the cusp curve})
\label{versdefcusp}
Let $f:X \subseteq \AA^2 \times \PP^2 \to \AA^2=Y$ be the family of projective curves 
\begin{equation}
\{(a,b,[X,Y,Z])\in \AA^2 \times \PP^2, ZY^2-X^3-aXZ^2-bZ^3=0\}.
\end{equation}
If $\Delta \subseteq \AA^2 $ is defined by 
$$\{(a,b)\in \AA^2, 4a^3 + 27b^2 = 0 \},$$ a stratification for the map is given by 
$$
\{(0,0) \subseteq \Delta \subseteq \AA^2 \}.
$$
Nevertheless, just as for the other points of $\Delta$, 
the inverse image of a one-dimensional complex disc through the origin has nonsingular total space.
Therefore $\Delta^2(f)=\emptyset$, and $\Delta^1(f)=\Delta$.
\end{example}

\begin{remark} 
\label{semismall-puzzle}
Let $f:Y \to X$ be a semismall-map, with $Y$ nonsingular, and assume $f(Y)=X$.
Define $X_i=\{x \in X \mbox{ such that } \dim Y_x=i\}$. A component $X_\alpha$ of $X_i$ supports a summand 
of $Rf_* \QQ_Y[\dim Y]$ if and only if $\dim X_\alpha=\dim X-2i$ (these are the {\em relevant strata}, see \cite{BM, dCM}).
It follows from \ref{thm:MS} that in this case $\overline{X_\alpha}\subseteq \Delta^{2i}(f)$. This means that there is no $2i-1$-dimensional disc through 
the general point of $X_\alpha$ with nonsingular inverse image of the right codimension. Notice that the consideration of the dimension of the fibre would give a much cruder 
estimate, namely that  there is no $i$-dimensional disc with nonsingular inverse image of the right codimension. For example, if a point $x$ is  a zero-dimensional relevant stratum, namely  $\dim Y_x=1/2 \dim Y$, there is no Cartier divisor through $x$ whose inverse image is a nonsingular Cartier divisor in $Y$. 
\end{remark}

\section{Theorem A: reformulations and applications}
Here we further spell out the consequences of Theorem A, and make some applications
to the problem of estimating cohomology in families. Throughout this section we work with complex varieties.

%

We now give some corollaries of Theorem \ref{thm:disclag}. 
We recall the relation between characteristic cycles, 
Euler obstructions, and conormal varieties.  For $Y$ smooth and $S \subseteq Y$,
we write $\overline {T_S^*Y}$ for the closure of the conormal bundle of the smooth locus of $S$. 
For a constructible function $f$ we write $CC(f)$ for the characteristic cycle of $f$.  (Usually
the characteristic cycle is defined for sheaf complexes or D-modules, but in any case it factors 
through the Grothendieck group and depends only on the underlying constructible function.) 

\begin{corollary} \label{cor:CC}
Let $X$ be a complex variety, and let $\xi$ be a constructible function on $X$.  
Let $\{ \Delta^{i,\alpha}\}_{\alpha}$ be the $i$-codimensional components of $\Delta^i(\xi)$. 
For $x_{i,\alpha} \in \Delta^{i,\alpha}$ a general point, $X \supset \DD^i \ni x_{i,\alpha}$ 
a general $i-$dimensional disc, and $l: \DD^i \to \CC$ a general linear form, 
define $\xi^{i,\alpha} = \Phi_l (\xi|_{\DD^i})(x_{i,\alpha})$.
Then 
$$CC(\xi) = \sum (-1)^i \xi^{i,\alpha} \overline{T_{\Delta^{i,\alpha}}^{\,*} X}$$
\end{corollary}
\begin{proof}
The only new ingredient beyond Theorem \ref{thm:disclag} is the provision of the
coefficients of the characteristic cycle; this however is essentially a tautology on 
the definition of characteristic cycle in terms of vanishing cycles (see the discussion in \cite{Br}, \S \, I.1). 
\end{proof}
The Brylinski-Dubson-Kashiwara index theorem \cite{BDK, G} asserts
$CC(\mathrm{Eu}_S) = (-1)^{\codim S} [\overline{T_S^* Y}]$, and we conclude

\begin{corollary}
With the hypotheses and notation of Corollary \ref{cor:CC}, 
$$\xi = \sum_{i,\alpha} \xi^{i,\alpha} \mathrm{Eu}_{\overline{\Delta^{i,\alpha}}}$$
\end{corollary}

We were originally interested in the setting where $f:Y \to X$ was a proper map of
complex varieties.  In this setting, we have: 

\begin{corollary} \label{cor:CC1}
Let $f: Y \to X$ be a proper map of complex varieties.  Let $\{ \Delta^{i,\alpha}\}_{i, \alpha}$ be the codimension $i$
components of $\Delta^i(f)$.  Then 
\begin{eqnarray*}
CC(f_* {\mathbf 1}_Y) & = & \sum_{i, \alpha} (-1)^i \xi^{i,\alpha} \overline{T_{\Delta^{i,\alpha}}^{\,*} X}\\
f_* {\mathbf 1}_Y & = & \sum \xi^{i,\alpha} \mathrm{Eu}_{\overline{\Delta^{i,\alpha}}} 
\end{eqnarray*}
The coefficients $\xi^{i,\alpha} $ are defined as in Corollary \ref{cor:CC}; though note by proper
base change they can instead be computed by studying vanishing cycles on $Y$. 
\end{corollary}
\begin{proof}
The essential statement here is that the codimension $i$
components of $\Delta^i(f) := \Delta^i(f_\dagger T_Y^* Y)$ 
are among the codimension $i$ components of 
$\Delta^i(f_* {\mathbf 1}_Y) := \Delta^i(SS(f_* 1_Y))$.   This follows formally from the fact \cite{KS} that
$SS(f_* {\mathbf 1}_Y) \subseteq f_\dagger SS({\mathbf 1}_Y) = f_\dagger T_Y^* Y$. 
\end{proof}

\begin{remark}
When $X$ is singular, in addition to ${\mathbf 1}_X$ we may consider 
$[\mathrm{IC}_X]$ for the IC sheaf in any perversity; we again have 
$\Delta^i(f_* [\mathrm{IC}_X]) \subseteq \Delta^i(f)$.  Additionally, 
by \cite{BBD}, the perverse cohomology sheaf ${}^p \mathscr{H}^k(Rf_* \mathrm{IC}_X)$ 
is a summand of $Rf_* \mathrm{IC}_X$, so we have 
$$\Delta^i([{}^p \mathscr{H}^k(Rf_* \mathrm{IC}_X)])\subseteq \Delta^i({}^p  \mathscr{H}^k(Rf_* \mathrm{IC}_X)) \subseteq \Delta^i(Rf_* \mathrm{IC}_X) \subseteq \Delta^i(f).$$
Thus these functions
also can be expanded in the Euler obstructions of the higher discriminants
{\em of the map $f$}.
\end{remark}

Finally, let $c_{SM}$ denote the Chern-Schwarz-MacPherson transformation from constructible functions to homology \cite{M}.  By definition
this is $c_{SM}(\mathrm{Eu}_{V}) = c^M(V)$, where $c^M$ is the Chern-Mather class. 
For $Y$ smooth we have $c_{SM}({\mathbf 1}_Y) = c(TY) \cap [Y]$; and the main result of \cite{M} is that $c_{SM}$ commutes with proper pushforward.
\begin{corollary}
Let $f:Y \to X$ be a proper map of complex varieties.  Let $\Delta^{i,\alpha}$ be the codimension $i$
components of $\Delta^i(f)$.  Then
\[f_* c_{SM}({\mathbf 1}_Y) = c_{SM}(f_* {\mathbf 1}_Y) = \sum \xi^{i,\alpha} c^M(\Delta^{i,\alpha}). \] 
\end{corollary}

The following is useful: 

\begin{lemma}
If $\fF$ is perverse, the closures of the $i$-codimensional components of  $\Delta^i(\fF)$ and  $\Delta^i([\fF])$ coincide.
\end{lemma}
\begin{proof}
Take a Whitney stratification for $\fF$. We need to check only at a generic point $y$ of a codimension $i$ stratum $S$. 
The restriction to a disc $\DD^i \ni y$ transverse to $S$ is, up to shifting by $[-i]$, perverse, hence so is the vanishing cycle 
along a general linear form.  The vanishing cycles in a small enough neighborhood of $y$ and
with respect to a general linear form are supported at a point; being perverse they therefore 
have (usual) cohomology sheaves in only one dimension.  Thus the vanishing of this vanishing cycle
is equivalent to the vanishing of its Euler characteristic. 
\end{proof}

\subsection{Remarks on the singular case}  \label{sec:sing} The above ideas make immediate
sense when $X$ is singular: one simply
chooses locally an embedding of $X$ into a smooth manifold.  
Below we comment on what can be done when $Y$ is singular.  
We omit the proofs, both because they are standard and because
although it is possible to set up the theory, we do not know in practice 
how to compute the higher discriminants.

Fix an embedding of $Y$ in a smooth manifold $\widetilde{Y}$.  
Let $\cS$ be the canonical Whitney stratification of $(\widetilde{Y}, Y)$ constructed in Chaper VI, \S 3 of \cite{T2}; 
let $T_{\cS}^* Y$ 
be the union of all 
conormals to strata in $Y$; by the Whitney conditions this is closed.  
  

Given a map $f:Y \to X$, fix an extension $\widetilde{f}: \widetilde{Y} \to X$.  Then we can
take $f_\dagger T^*_Y Y := \widetilde{f}_\dagger (T^*_Y Y)$; 
which gives
a conical subvariety of $X$.  We can again define $\Delta^i(f):= \Delta^i(f_\dagger T^*_Y Y)$.  

Lemma \ref{lem:discfun} has the following variant: $\Delta^i(f)$ is the locus of $y \in Y$ such that
there is no $(i-1)$ germ of a complete intersection $\DD^{i-1}$ passing through $y$ for which the inclusion  
 $f^{-1}(\DD^{i-1}):=(X \times_Y \DD^{i-1}) \hookrightarrow X$ is normally nonsingular, i.e. it admits a neighborhood in $X$ homeomorphic to a bundle over itself and ${\rm codim}(\DD^{i-1}, Y) ={\rm codim}(f^{-1}(\DD^{i-1}), X)$. 

With these definitions, Theorem C again holds.  In Theorem A, it is natural to replace $\Q_Y$ 
with the intersection cohomology sheaf with respect to some perversity; the result holds in all
of these cases because, with the above definitions, using \cite[\S 5.4.1]{GM2} we have 
$SS(IC_Y) \subseteq T_Y^* Y$ hence 
$SS(f_* IC_Y) \subseteq f_\dagger(SS(IC_Y)) \subseteq f_\dagger T_Y^* Y$ hence
$\Delta^i( SS(f_* IC_Y) ) \subseteq \Delta^i(f)$.  Theorem B will hold with $\Q_Y$ replaced
by the middle-perverse intersection cohomology sheaf.

\section{Theorem B in characteristic p} \label{sec:support}

Over positive characteristic fields, easy examples with inseparable morphisms show that 
the estimate (\ref{codest}) on the codimension of the higher discriminants may not hold.
Furthermore, even assuming that the estimate (\ref{codest}) holds, 
the methods used in the previous sections cannot be extended.
In the recent preprint \cite{B}, Beilinson defines a notion of singular support for \'etale constructible sheaves on an algebraic variety $X$ over an arbitrary field, as a closed subset of $T^*X$, 
and proves that every irreducible component has dimension $\dim X$. However, the examples given in \cite{B}, 
show that this is not as well behaved as in characteristic zero, in particular the singular support is not 
necessarily  a union of conormals to subvarieties of $X$.    
Here we give another proof of Theorem B, for the direct image of the constant sheaf on a nonsingular variety,
valid also over positive characteristic fields, under the supplementary hypothesis that the codimension estimate (\ref{codest}) holds.

\medskip
Throughout we employ the formalism of perverse sheaves \cite{BBD}.  
For $Z$ an algebraic variety and $K$ a constructible complex $K \in \mathrm{D}^b_c(Z) $, 
we denote by $^p\!\!{\mathscr H}^i(K)$ its $i$-th perverse cohomology sheaf.  We adopt the usual shift convention: 
if $Z$ is smooth, then $\QQ_Z[\dim Z]$ (or $\QQ_{l , \,Z}[\dim Z]$ in the \'etale setting) is perverse.  We work in some setting 
where a formalism of weights is available; either with $\ell$-adic sheaves over
a finite field $\FF_q$, or with mixed Hodge modules over $\CC$. We say that a complex  $K \in \mathrm{D}^b_c(Z) $ is {\em pure semisimple of weight $k$} if it is isomorphic to the direct sum of its perverse cohomology sheaves, and, for every $i$,  
$^p\!\!{\mathscr H}^i(K)$ is a semisimple pure perverse sheaf of weight $i+k$.

We reformulate the definition of higher discriminants of a map in this context:

\medskip
\noindent 
{\em Terminology:} Let $X$ be a nonsingular variety. 
Given $x \in X$,  by ``a $k-$dimensional disc through $x$" $\DD^{k} \hookrightarrow X$, 
we mean a germ of nonsingular $k-$dimensional subvariety passing through $x$.

\begin{definition} \label{def:homologicaldiscsing} 
 Let $f: Y \to X$ be a proper map of nonsingular algebraic
varieties. A $k-$dimensional disc $\DD^{k} \hookrightarrow X$ through $x \in X$ is {\em transverse} to $f$
if 
$ f^{-1}(\DD^{k})  $ is  nonsingular along  $f^{-1}(x)$ and 
$\mathrm{codim}(f^{-1}(\DD^{i-1}),Y)=\mathrm{codim}(\DD^{i-1},X)$.
\end{definition}

\begin{definition}
 Let $f: Y \to X$ be a proper map of nonsingular algebraic
varieties.
\begin{equation}
\begin{split}
\Delta^i(f)= \{ x \in X \text{ s.t. there is no }\DD^{i-1} \hookrightarrow X \text{ through } x 
\text{ transverse to } f \}.
\end{split}
\end{equation}


\end{definition}	

\begin{remark}
The definition generalizes immediately to the case of $X$ locally embeddable 
into a smooth variety; all the results hold mutatis mutandis for such $X$. 
\end{remark}

\subsection{Supports are discriminants}

The decomposition theorem of \cite{BBD} asserts that if $f: Y \to X$ is proper and $Y$ is nonsingular, 
then $Rf_* \QQ$ (or $Rf_* \QQ_\ell$) is pure semisimple.  
In other words, there are some nonsingular locally closed subvarieties 
$V_i$
carrying semi-simple local systems $L_i$ so that $Rf_* \QQ = \bigoplus \mathrm{IC}(V_i, L_i)[d_i]$,
for appropriate $d_i \in \ZZ$. 

\begin{theorem}\label{thm:MS}Let $f:Y \to X$ be a projective map between 
algebraic varieties, with $Y$ nonsingular.   
Assume $\codim \Delta^i(f) \ge i$. 
Then every summand
of $Rf_* \QQ_\ell$ is supported on the closure of an $i$-codimensional component of some discriminant 
$\Delta^i(f)$.\end{theorem}

Often, even in characteristic $p$, it is possible to just directly make the
tangent space calculation determining the $\Delta^i(f)$ and directly show that
$\Delta^i(f)$ has codimension $\geq i$. For an application, see \cite{MSV}. 

\vspace{2mm}
We turn to giving the proof of Theorem \ref{thm:MS}.  First we develop some preliminary notions. 

\begin{definition}
\label{classS}
We write $\mathfrak{S}(Z) \subseteq \mathrm{D}^b_c(Z)$ for the set of complexes $K$ which are (1) pure semisimple, 
and moreover (2) are symmetric in the sense that
$^p\!\!{\mathscr H}^i(K)\simeq \, ^p\!\!{\mathscr H}^{-i}(K)(-i)$ for every $i \in \ZZ$. 
\end{definition}


\begin{remark}
\label{trivialities}
We have:
\begin{enumerate}
\item
If $K, K[r] \in \mathfrak{S}(Z)$, then either $K = 0$ or $r = 0$. 
\item If $K=K_1\oplus K_2$, and two of the three complexes are in $\mathfrak{S}(Z)$, then
so is the third.
\item If $K \in \mathfrak{S}(Z)$ and $K_\Lambda$ is the sum of the summands of $K$ with support exactly equal to 
$\Lambda \subseteq Z$, then $K_\Lambda \in \mathfrak{S}(Z)$. 
\end{enumerate}
\end{remark}

\begin{lemma} \label{lem:div} \cite[Corollaire 4.1.12]{BBD}  If $P$ is a {\em simple} perverse sheaf on $Z$, 
and $i: D \hookrightarrow Z$ is a Cartier divisor, then either 
\begin{enumerate}
  \item $\mathrm{supp} P \subsetneq D$ and $i^*P[-1]$ is perverse, or
  \item $\mathrm{supp} P \subseteq D$ and $i^*P$ is perverse.
\end{enumerate}
\end{lemma} 

\begin{corollary}\label{restrictionDT}
Let $Z$ be a  variety, let
$i: D \hra Z$ denote the closed immersion of a Cartier divisor.
If $K \in \mathfrak{S}(Z)$ and $i^* K[-1] \in \mathfrak{S}(D)$, then
$D$ does not contain the support of any summand of $K.$
\end{corollary}
\begin{proof}
Write $K=K' \oplus i_*K''$, where $i_*K''$ is the direct sum of all the summands whose support is contained in $D$.
By the remark, $K', i_* K'' \in \mathfrak{S}(Z)$.  By hypothesis, $i^*K$ and hence $i^*K'$ and $i^* i_* K'' = K''$ have 
semisimple perverse cohomology sheaves.  By Lemma \ref{lem:div}, (up to Tate twists) 
\[  {}^p\!\mathscr{H}^k(i^* K' [-1]) = i^* {}^p\! \mathscr{H}^k(K') [-1] \simeq  i^* {}^p\! \mathscr{H}^{-k}(K')[-1]  =   {}^p\! \mathscr{H}^{-k}(i^* K' [-1]) \]
and so $i^* K'[-1] \in \mathfrak{S}(D)$.  Thus since $i^* K[-1] = i^* K'[-1] \oplus K''[-1] \in \mathfrak{S}(D)$, we have $K''[-1] \in \mathfrak{S}(D)$. 
On the other hand obviously $K'' \in \mathfrak{S}(D)$, which is a contradiction unless $K'' = 0$. 
\end{proof}

\begin{proof}(of Theorem \ref{thm:MS}) 
As $f$ is projective and, by smoothness of $Y$, the sheaf $\cF:=\QQ_\ell[\dim Y]$ is pure and perverse,
it follows from the decomposition theorem and the relative Hard Lefschetz theorem that $Rf_* \cF \in {\mathfrak S}(Y)$. 

We first show that
the $0$-dimensional supports of the summands of $Rf_* \cF$ are contained in $\Delta^{\dim Y}(f)$.
Indeed, if $y \notin \Delta^{\dim Y}(f, \cF)$, then by definition there exists a divisor $i: D \to Y$, containing $y$, such that
$f^{-1}(D)$ is nonsingular in a neighborhood of $f^{-1}(y)$.  By proper base change
$ R(f|_{D})_* (i^* \cF[-1]) =i^*Rf_* \cF[-1]$ on a neighborhood of $y$.  
By Corollary \ref{restrictionDT}, there is no summand of $Rf_* \cF$ supported on $y$.  

In the general case, let $Z$ be the support of a simple summand $\cG$ of $Rf_*\cF $, and denote by $k$ its codimension.
 By the 
codimension hypothesis $Z \setminus \Delta^{k+1}(f)$ is dense in $Z$.
Let $z \in Z \setminus \Delta ^{k+1}(f)$: it will suffice to show $z \in \Delta^k(f)$. 
Indeed, suppose otherwise:  By the assumption on $Z$, we know that for some $k-1$ dimensional disc $\DD^{k-1}$ through $z$, 
we have  $f^{-1}(\DD^{k-1})$ is nonsingular.  Since nonsingularity is an open condition, there is a $k$ dimensional disc $\DD^{k}$ 
through $z$ containing $\DD^{k-1}$ such that $f^{-1}(\DD^{k})$ is nonsingular.
Consider  the restriction 
$f_|: f^{-1}({\DD^k}) \to \DD^k$.  By proper base change there
is a non-zero summand  of $Rf_{|*} \cF[k-\dim Y]|_{{\DD}^k}$ supported at $z$, but also a codimension one disc through $z$ transverse to $f_|$ and we find a contradiction from what proved in the first step.  
\end{proof}

\begin{remark}
In \cite[\S 5]{MS}, we proved (but did not state) a weaker form of Theorem \ref{thm:MS}. 
\end{remark}

\begin{remark}
One might hope for a more precise form of Theorem \ref{thm:MS}, holding for the direct image of a pure perverse sheaf
$\cF$ on $Y$, in which the role of the higher discriminants $\Delta^i(f)$ is played by
the "discriminants of pure perversity" $\Delta^i_{p}(f, \cF)$, defined as follows: Say that a disc through $x\in X$ is transverse if 
$\cF|_{f^{-1}(\DD^{i-1})}[-{\rm codim}(\DD^{i-1}, X)]$ is pure perverse. Then define $\Delta^i_{p}(f, \cF)$ to be 
the locus of $x \in X$ thorugh which there is no $(i-1)$ dimensional transverse disc $\DD^{i-1} \to X$. 
Unfortunately it is not clear whether these discriminants have the required properties to make the argument of \ref{thm:MS} work.
Clearly $\Delta^i_{p}(f, \QQ) \subseteq \Delta^i(f)$ and it is easy to find examples in which inclusion is strict, e.g. with finite maps, but, on the other hand, the determination of these discriminants looks quite hard.
\end{remark}

\subsection{Further techniques for determining supports}
We recall some definitions from \cite{CL2}, \S 6: 
\begin{definition}
Given a semisimple complex $K=\bigoplus \mathrm{IC}(V_i, L_i)[d_i]$, we set 
$$
\mathrm{Socle}(K)=\{V_i \text { t.c. } L_i \neq 0\}.
$$
\end{definition}  

Here, the $V_i$'s are nonsingular locally closed subvarieties carrying local systems $L_i$,
and we call these $V_i$'s (or more precisely their generic points) the {\em supports} of $K$.

Often, some of the higher discriminants may turn out not to support summands of $Rf_* \cF$.  
To see whether this is the case requires a single numerical calculation at a generic point of the support of each
summand.  

For instance, working over $\CC$, when $X$ is nonsingular, and $\cF = \QQ_X$, 
we have $Rf_* \QQ_X = \bigoplus \mathrm{IC}(V_i, L_i)[d_i]$ by the decomposition theorem.  
By \cite{saito1, saito2, saito3}, there is a canonical Mixed Hodge structure on the stalks of the 
cohomology sheaves of every direct summand $K$ in this decomposition of $Rf_* \mathrm{IC}_X$,
compatible with the canonical Mixed Hodge structure on the cohomology of the fibres of $f$.

For a weight filtration $W$ on a complex of vector spaces $V^{\bullet}$, we
write the weight polynomial 
\[
\mathfrak{w}(V^{\bullet}) := \sum_{i,j} t^i (-1)^{i+j} \dim \mathrm{Gr}_W^i 
\mathrm{H}^j(V^{\bullet}).
\]  
For a variety $Z$, we abbreviate 
$\mathfrak{w}(Z)$ for $\mathfrak{w}(\mathrm{H}^*_c(Z))$.

\begin{lemma}
Assume $K, K'$ are pure complexes of geometric origin on an algebraic variety $X$, and
in particular semisimple. If $i: K' \to K$ is a monomorphism such that 
${\mathfrak w}(K'_x) = {\mathfrak w}(K_x)$ at the generic point of each support of $K$, then
$i$ is an isomorphism. 
\end{lemma}
\begin{proof}
The existence of such a monomorphism implies $K, K'$ are of the same weight, we take it to be zero. 
Any monomorphism in a triangulated category splits, so there exists some  $K''$ such that 
$K \cong K' \oplus K''$.  Any simple constituent of $K''$ restricts, on a 
nonempty open subset of its support to a shifted pure local system $L[-d]$.  Since $K$ was
pure of weight zero, the stalk of $L$ is pure of weight $d$ and 
contributes the positive quantity $(-1)^{d+d} t^d \dim L$ to the weight polynomial. 
\end{proof}

\begin{corollary}\label{weightcriterion}
Let $f: Y \to X$ as in Theorem \ref{thm:MS},  and assume furthermore $Y$ nonsingular. 
Let $K$ be a pure complex, and let $i: K \to Rf_* \QQ_Y$ be a monomorphism in $\mathrm{D}^b_c(X)$. 
If $ {\mathfrak w}(K_y) = {\mathfrak w}(X_y)$ at the general point of every higher discriminant, then $i$ is an isomorphism.
\end{corollary}

We used this method in \cite{MS}.  In its sequel \cite{MSV} we require the ability to make a similar
argument in the absence of a morphism $i$. 
This can be done after passing to a finite
field, using the following incarnation of the Cebotarev theorem which we now describe: 

Let $X$ be an algebraic variety defined over a finite field ${\FF}_q$, and let $K, K'$ be pure complexes of $\ell$-adic sheaves on $X$.
Denote by $\ov{X}, \ov{K}, \ov{K'}$ the corresponding objects obtained by base change to the algebraic closure  $\overline{\FF}_q$.
For every closed point $x \in X$, and  geometric point $\ov{x} \in X(\ov{\FF})$ lying over it,  we have a Frobenius map $\mathrm{Fr}:\ov{K}_{\ov{x}}\to \ov{K}_{\ov{x}}$. We set 
$$Tr (\mathrm{Fr}_{x}, K):=\sum (-1)^k  Tr \left(\mathrm{Fr} :{\mathcal H}^k(\ov{K}_{\ov{x}}) \to {\mathcal H}^k(\ov{K}_{\ov{x}})\right).$$
\begin{proposition}
Assume 
\begin{enumerate}
\item 
$ \mathrm{Socle}(\ov{K})= \mathrm{Socle}(\ov{K'})=\{V_i\}_{i \in \Lambda}$
\item
for every $i \in \Lambda$, there is an open subset $V_i^o \subset V_i$ such that, 
for every closed point $x_i \in V_i^o$,  
$Tr (\mathrm{Fr}^n_{x_i}, K)=Tr (\mathrm{Fr}^n_{x_i}, K')$ for every $n$.
\end{enumerate}  
Then $\ov{K}\simeq \ov{K'}$. 
\end{proposition}

\begin{corollary}
Let $f:Y \to X$ a projective map, with $Y$ nonsingular.
Assume that for every $i$ we have 
$\codim \Delta^i(f) \ge i$,  and let $\{V_\alpha\}_{\alpha \in I}$ 
be the set of codimension $i$ irreducible components of $\Delta^i(f)$ for all $i$'s.
Let $K$ a semisimple pure complex of $\ell$-adic sheaves on $Y$ with
$\mathrm{Socle}(\ov{K}) \subseteq \{V_\alpha\}_{\alpha \in I}$.
Assume for every $\alpha \in I$ there is an open subset $V_\alpha^o \subset V_\alpha$ with the property that, 
for every closed point $x_\alpha \in V_\alpha^o$ and for every $n$,  
\begin{equation}\label{eqtr}
Tr (\mathrm{Fr}^n_{x_\alpha}, Rf_*  \QQ_\ell)=Tr (\mathrm{Fr}^n_{x_\alpha}, K).
\end{equation}
Then $\ov{K}\simeq \ov{Rf_* \QQ_\ell}$. 

\end{corollary}
 Notice that by the Grothendieck-Lefschetz trace formula, the left hand side in equation (\ref{eqtr})
 reduces to counting points on $f^{-1}(x_\alpha)$.

\section{Stalks and Bounds} \label{sec:sab}

In this section we recall results of D. Massey \cite{Ma} on the relation between
stalks of perverse sheaves, local polar varieties, and the characteristic cycle, and draw corollaries
by plugging in our newfound understanding of the characteristic cycle in terms of higher discriminants. 
We work in the complex analytic setting.

\vspace{2mm}
Let $K$ be a constructible sheaf complex or a constructible function.
Fix $p \in Y$, and choose a general system of 
coordinates $(y_1,\ldots, y_{\dim Y})$ at $p$.  We write $\Xi_{y_i}$ for the restriction to $y_i = 0$. 
For the remainder of the section, we {\em always} take vanishing cycles, nearby cycles, and 
restriction with respect to these coordinates and in
the coordinate order, and thus drop the subscripts.  I.e.,
$\Psi \Phi \Psi \Xi K := \Psi_{y_4} \Phi_{y_3} \Psi_{y_2} \Xi_{y_1} K$. 

\begin{remark}
By definition, $p \notin \Delta^i(K)$ if and only if $(\Phi \Xi^{\dim Y - i} K)_p = 0$. 
\end{remark}

\begin{proposition}
We have $p \notin \Delta^{\ge r} (K)$ if and only if
$(\Phi \Psi^{\le \dim Y - r} K)_p = 0$.
\end{proposition}
\begin{proof}
We treat the case where $K$ is a constructible complex; the case of constructible
functions is similar but easier.  

For $r = \dim Y$ the statement holds by definition.  Suppose by induction it holds
for $r= n$, we show it for $r = n-1$.  

$(\Longrightarrow)$ We have by hypothesis  $\Phi \Xi^{\le \dim Y - n + 1}_p = 0$ 
and by induction  
$(\Phi \Psi^{\le \dim Y - n} K)_p = 0$.  We must show $(\Phi \Psi^{\dim Y - n + 1} K)_p = 0$. Since $(\Phi \Psi^{\le \dim Y - n} K)_p = 0$, 
the exact triangle (\ref{phipsitri}) for the complex 
$ \Psi^{\le \dim Y - n} K$ and the function $y_{\dim Y -n+1}$ gives an isomorphism
$$(\Xi \Psi^{ \dim Y - n} K)_p \cong (\Psi^{\dim Y - n + 1} K)_p;$$
applying $\Phi$ we have the isomorphism
$$ (\Phi \Xi \Psi^{ \dim Y - n} K)_p \cong (\Phi \Psi^{\dim Y - n + 1} K)_p ,$$
 and it suffices to show the vanishing of
$(\Phi \Xi \Psi^{ \dim Y - n} K)_p$.  We  repeat the process: 
$$(\Phi \Xi \Phi \Psi^{ \dim Y - n-1} K)_p \to (\Phi \Xi^2 \Psi^{ \dim Y - n-1} K)_p \to (\Phi  \Xi \Psi^{\dim Y - n} K)_p 
\xrightarrow{[1]}.$$
Again the first term vanishes.  Iterating we arrive at
$$(\Phi \Psi^{\dim Y - n + 1} K)_p \cong (\Phi \Xi \Psi^{ \dim Y - n} K)_p \cong \cdots \cong 
(\Phi \Xi^{ \dim Y - n + 1} K)_p = 0.$$

$(\Longleftarrow)$  We have by hypothesis $(\Phi \Psi^{\le \dim Y - n + 1} K)_p = 0$
and by induction $\Phi \Xi^{\le \dim Y - n }_p = 0$.  We must show $(\Phi \Xi^{\dim Y - n + 1} K)_p = 0$.  We have the exact triangle
$$(\Phi \Phi \Xi^{\dim Y - n} K)_p \to (\Phi \Xi^{ \dim Y - n + 1} K)_p \to (\Phi \Psi \Xi^{\dim Y - n} K)_p \xrightarrow{[1]}.$$
As before the first term vanishes, and iterating this process gives
$$(\Phi \Xi^{\dim Y - n + 1} K)_p \cong (\Phi \Psi \Xi^{ \dim Y - n} K)_p \cong \cdots \cong 
(\Phi \Psi^{ \dim Y - n + 1} K)_p = 0.$$
\end{proof}



We now recall some results of Massey \cite{Ma}.  For any complex $K$, 
the nearby-vanishing triangle of $\Psi^k K$ gives rise to a long exact sequence,
of which the following is a fragment: 
$$\hH^{-k-1}(\Psi^{k+1} K)_p  \xrightarrow{\beta_{-k-1}} \hH^{-k}(\Phi \Psi^k K)_p \xrightarrow{\gamma_{-k}} \hH^{-k}(\Psi^k K)_p.$$
As the composition $\gamma_{-k} \beta_{-k-1} = 0$, the following is a complex: 
$$\cdots \to \hH^{-k-1}(\Phi \Psi^{k+1} K)_p \xrightarrow{\beta_{-k-1} \gamma_{-k-1}} \hH^{-k}(\Phi \Psi^k K)_p \xrightarrow{\beta_{-k} \gamma_{-k}} 
\hH^{-k+1}(\Phi \Psi^{k-1} K)_p \to \cdots$$

\begin{theorem} \label{Masseycomplex} \cite[Thm. 5.4]{Ma} If $K$ is perverse, then the complex
$$\hH^{-\dim Y}(\Psi^{\dim Y} K)_p \xrightarrow{\beta_{-\dim Y}}
\hH^{-\dim Y + 1}(\Phi \Psi^{\dim Y - 1} K)_p 
\cdots \xrightarrow{\beta_{-1} \gamma_{-1}}  \hH^{-1}(\Phi \Psi K)_p \xrightarrow{\beta_0 \gamma_0} \hH^0(\Phi K)_p $$
has the same cohomology as $K_p$. 
In the above, each term occupies the same cohomological degree as the superscript on $\hH$.
\end{theorem}
\begin{proof}\footnote{Massey leaves the final verifications in the proof as an exercise, which we carry out here.}
It follows easily from the existence of Whitney stratifications that,
fixing in advance a point $p \in Y$ and a complex $K$, one has for a sufficiently small $\epsilon$-ball that
$\mathrm{Supp}\,\Phi K \cap B_\epsilon(p) = p$.  On the other hand, the vanishing cycles 
of a perverse sheaf are perverse, so if $K$ is perverse then $\Phi K|_{B_{\epsilon}(p)}$ is a perverse sheaf supported
at a point, i.e., just a skyscraper sheaf in homological degree zero.   We get the following exact sequence from
the nearby-vanishing triangle:
$$0 \to \hH^{-1} (K)_p \to \hH^{-1}(\Psi K)_p \xrightarrow{\beta_{-1}} \hH^0(\Phi K)_p \xrightarrow{\gamma_0} \hH^0(K)_p \to 0 .$$
The zero on the left is because $\Phi K$ is a skyscraper in degree zero; the zero on the right
is because $\Psi K [-1]$ is perverse.  We also learn 
$\hH^i(\Psi K)_p \cong \hH^i(K)_p$ for $i < -1$. 
Since $\Psi K[-1]$ is perverse, we can iterate this procedure,
obtaining more generally the exact sequence
$$0 \to \hH^{-k-1} (\Psi^k K)_p \to \hH^{-k-1}(\Psi^{k+1} K)_p  \xrightarrow{\beta_{-k-1}} \hH^{-k}(\Phi \Psi^k K)_p \xrightarrow{\gamma_{-k}} \hH^{-k}(\Psi^k K)_p \to 0$$
and
$\hH^i(\Psi^{k+1} K)_p \cong \hH^i(\Psi^k K)_p  \cong \cdots \cong \hH^i(K)_p$ for $i < -(k+1)$. 
In particular we learn $\hH^{-k-1}(\Psi^k K)_p \cong \hH^{-k-1}(K)_p$, so we may rewrite the above sequence as 
$$0 \to \hH^{-k-1} (K)_p \to \hH^{-k-1}(\Psi^{k+1} K)_p  \xrightarrow{\beta_{-k-1}} \hH^{-k}(\Phi \Psi^k K)_p \xrightarrow{\gamma_{-k}} \hH^{-k}(\Psi^k K)_p \to 0.$$
In particular, $\hH^{-i}(K)_p = \rm{ker}\, \beta_{-i}$.  When $i=\dim Y$ we are done; otherwise since $\gamma_{-i}$ is surjective, it induces
$\overline{\gamma}_{-i}: \ker \beta_{-i} \gamma_{-i} \twoheadrightarrow \ker \beta_{-i} = \hH^{-i}(K)_p$.  Evidently $\ker \overline{\gamma}_{-i} = \ker
\gamma_{-i}$.  On the other hand, since
$\gamma_{-i-1}$ is surjective, $\rm{im}\, \beta_{-i-1} \gamma_{-i-1} = \rm{im}\, \beta_{-i-1} = \ker \gamma_{-i}$.  Thus we conclude
$\ker \beta_{-i} \gamma_{-i} / \rm{im}\, \beta_{-i-1} \gamma_{-i-1} \cong \hH^{-i}(K)_p$. 
\end{proof}

\begin{corollary} \label{cor:stalks} 
	Let $K$ be perverse and $p \notin \Delta^{\ge r}(K)$.  Then  $\hH^{\ge r - \dim Y}(K)_p = 0$. 
\end{corollary} 

\begin{remark} The first term of the complex of Theorem \ref{Masseycomplex} looks different from the rest; this annoyance can be avoided by assuming
$\codim\, \rm{Supp}\,\, K > 0$ so that $\Phi^{\dim Y} K = 0$.  This may be ensured at no cost 
by embedding $Y$ into a larger space.  We assume this is the case henceforth. 
\end{remark}

As $\Phi \Psi^k K[-k]$ is a perverse sheaf supported on a point and thus a skyscraper
in degree zero, 
\[\dim \hH^{-k}(\Phi \Psi^k K)_p =   \chi( (\Phi \Psi^k K[-k])_p) = (-1)^k \Phi \Psi^k [K](p).\]

We turn to the question of how to compute this Euler characteristic.  Recall the local polar varieties
\cite{LT}: 
for $p \in S \subseteq Y$, we write $S^{\mathrm{sm}}$ for the smooth locus of $S$ and
the polar varieties of $S$ (with respect to a choice of general coordinates near $p$)
are by definition
$$\Gamma_i^S:= \mbox{the closure of the union of the $i$-dimensional components of }
\mathrm{Crit}(S^{\mathrm{sm}} \xrightarrow{y_1,\ldots,y_{i+1}} \CC^{i+1}),$$
the set of critical points of the map.
The following is a reformulation of \cite[Thm. 5.1.1]{LT}:

\begin{theorem} 
Fix a variety $V \subseteq Y$, a point $p\in V$, and general coordinates around $p$. 
Then $$ (-1)^i (\Phi \Psi^i \mathrm{Eu}_V)(p) = (-1)^{\dim V} \mathrm{mult}_p\, \Gamma_i^V.$$
\end{theorem}
\begin{proof}
As always, genericity of the coordinates means that, in a sufficiently small neighborhood,
vanishing cycles are supported at $p$.   In particular the hyperplane formula $\Phi \rm{Eu} = 0$ 
holds not just at $p$, but on a sufficiently small neighborhood. 

We write $H^n$ for the zero locus of the first $n$ coordinates.  Since the coordinates are generic,
and in particular transverse to all the strata in a Whitney stratification except possibly $\{p\}$,
we have on $H$ the following equality of constructible functions, for some constant $c_1$:
$$\mathrm{Eu}_{X \cap H} = \Xi \mathrm{Eu}_X + c_1 \mathbf{1}_p = \Psi \mathrm{Eu}_X + c_1 \mathbf{1}_p.$$ 
Applying the hyperplane formula again, we have 
$ \mathrm{Eu}_{X \cap H}(p) = (\Psi \Psi \mathrm{Eu}_X)(p)$
and similarly $\mathrm{Eu}_{X \cap H^n}(p) = (\Psi^{n+1} \mathrm{Eu}_X)(p)$.   In other words, 
$$(\Phi \Psi^n \mathrm{Eu})(p) = (\Phi^{n} \mathrm{Eu})(p) - (\Phi^{n+1} \mathrm{Eu})(p) = 
\mathrm{Eu}_{X \cap H^{n-1}}(p)  - \mathrm{Eu}_{X \cap H^n}(p).$$
By \cite[Rem. 5.1.6]{LT} which is a reformulation of \cite[Thm. 5.1.1]{LT}, we have 
$$\mathrm{Eu}_{X \cap H^{n-1}}(p) - \mathrm{Eu}_{X \cap H^{n}}(p) = (-1)^{\dim (X \cap H^{n-1}) - 1} \mathrm{mult}_p \Gamma_1^{X \cap H^{n-1}} = 
(-1)^{\dim X - n} \mathrm{mult}_p \Gamma_n^{X} ,$$
the last equality being \cite[Cor. 4.1.6]{LT}.  This completes the proof.
\end{proof}

\begin{corollary}
Let $\xi$ be a constructible function on $Y$ with characteristic
cycle given by
$$CC(\xi) = \sum_S (-1)^{\codim S} \xi^S [\overline{T^*_S Y}].$$
Then $$(-1)^{i} (\Phi \Psi^i \xi)(p) = \sum_S (-1)^{\dim S} \xi^S \mathrm{mult}_p\, \Gamma_i^S .$$
\end{corollary}
\begin{proof}
This is equivalent to the theorem by the index formula \cite{BDK} and linearity. 
\end{proof}

\begin{remark} 
The corollary originates in 
\cite[Cor. 4.10, Thm. 7.5]{Ma}, where it was also proven using \cite{LT}, but appears 
to depend on a lengthy discussion of `characteristic polar cycles'. 
\end{remark}

\begin{corollary} \label{cor:bounds}
Let $K \in D^b(Y)$ be perverse, $\Delta^{i, \alpha}$ the $i$-codimensional components of 
$\Delta^i([K])$,  and $c^{i,\alpha}$ the coefficients 
such that 
\[
[K] = \sum c^{i, \alpha} \mathrm{Eu}_{\overline{\Delta^{i, \alpha}}}\,.
\]
Then for any $p \in Y$, we have 
$$\dim \hH^{-k}(K_p) \le \sum_{i, \alpha} c^{i,\alpha} \mathrm{mult}_p \Gamma_k^{\Delta^{i,\alpha}} .$$
\end{corollary}

\begin{remark}
One can give the usual more precise `Morse inequalities' which relate the dimensions
of the homology groups to the dimension of the terms in the complex. 
\end{remark}

Note that in particular, Corollaries \ref{cor:stalks} and \ref{cor:bounds} can be applied to $K = \, ^p\!\!{\mathscr H}^i( R f_* \mathrm{IC}_X)$, 
in which case we know that $\Delta^i([^p\!\!{\mathscr H}^i( R f_* \mathrm{IC}_X)]) \subseteq \Delta^i(^p\!\!{\mathscr H}^i( R f_* \mathrm{IC}_X)) \subseteq \Delta^i(f)$.

\section{Higher discriminants of algebraic completely integrable systems} \label{sec:dima}
In this section we show how one can determine the higher discriminants of an algebraic 
completely integrable system. We are grateful to  D. Arinkin for explaining us the following argument, see \cite{AF,N1}. 
 We work in the following set-up (see \cite{DM}):
 \begin{enumerate}
\item 
$X$ is a complex nonsingular quasi-projective variety, endowed with an algebraic symplectic form $\Omega$.
\item
$h:X \to Y$ is a flat projective map to a nonsingular $d-$dimensional affine variety $Y$ with $2\dim Y =\dim X$.
The restriction of $\omega$ to the smooth locus of any fiber is zero.

\item 
there exists a commutative group-scheme $G \to Y$ acting on $X \to Y$, with $\dim G=\dim X$, making $h$ into a completely integrable system:
For every point $x\in X$ let $dh_x:T_xX\to T_{h(x)}Y$ be the differential, $^t dh:  T_{h(x)}^{\vee}Y \to T_x^\vee X \stackrel{\Omega}{\cong}T_xX$
its adjoint. Then we have an isomorphism $\mathrm{Lie} \,(G_{h(x)})\cong T_{h(x)}^\vee Y$, and the derivative of the action is given by the linear map
\[
\mathrm{Lie} \,(G_{h(x)}) \cong T_{h(x)}^\vee Y \stackrel{\,^t \!dh}{ \longrightarrow} T_x^\vee X \stackrel{\Omega}{\cong}T_xX
\] above.
\end{enumerate}

By the Theorem of Chevalley, for every $y \in Y$ the neutral component $G_y^o$ of the commutative algebraic group 
$G_y$ is an extension 
\[
1\to H_y \to G_y^o \to A_y \to 1
\]
of an abelian variety $A_y$ by a connected affine algebraic group $H_y$, 
whence an upper-semicontinuous function 
\[ 
\delta : Y \to {\mathbb N}; \, \, \delta(y)=\dim H_y. 
\]
In particular, for every $y\in Y$ we have the subspace $\mathrm{Lie}(H_y) \subseteq \mathrm{Lie} \,(G_{y})\cong T_{y}^\vee Y.$
We assume that $\delta^{-1}(0)\neq \emptyset$ and that the stabilizers of the points in $x$ are affine subgroups.
This last assumption is not particularly restrictive, in view of the following 
\begin{lemma}\label{stab}
Claim: Suppose a connected algebraic group $G$ acts on a connected
variety $Z$. Then either the stabilizers of all points of $Z$ are affine,
or the stabilizers of all points of $Z$ fail to be affine.
\end{lemma}

\begin{proof} We may assume $Z$ irreducible. Suppose there is a
point $x\in Z$ whose stabilizer $S(x)<G $ is not affine. The stabilizer acts
on the local ring $\mathcal{O}_{Z,x}$; since $S(x)$ is not a linear group, this
representation must have a kernel $H(x)$. 
Since $H(x)$ acts trivially on the local ring $\mathcal{O}_{Z,x}$, it acts trivially on
the entire $Z$, and we are left to show that $H(x)$ is not affine.

Clearly the quotient $S(x)/H(x)$ is affine, having a faithful linear representation.  
Hence, the kernel $H(x)$ cannot be affine, otherwise $S(x)$ would be an extension of
an affine group by an affine group, and therefore affine itself.
\end{proof}
We note that in many important cases, such as the Hitchin system, 
the action of the group is free on an open set, therefore 
\begin{proposition}
We have
\[
\Delta^i(h)=\{y \in Y, \, \hbox{ such that } \delta(y)\geq i\}
\]
which is a purely codimension $i$ subset.
\end{proposition}
\begin{proof}

For $ x\in X$, let $y=h(x)$, and let  $S_x \subseteq G_y $ denote the neutral connected component of its stabilizer. Its
Lie algebra $\mathrm{Lie}(S_x) \subseteq {\mathrm Lie}(G_y)= T^{\vee}_yY$. 
From the definition of hamiltonian action it follows that
 \[
 dh(T_xX)= \mathrm{Lie}(S_x)^{\perp}\subseteq T_yY
 \] 
 is the orthogonal complement of
this subspace $\mathrm{Lie}(S_x)\subseteq T^\vee_yY$. By hypothesis, 
$S_x$ is affine, hence  contained in the maximal affine $H_y$. This implies that the images
$dh(T_xX)$ for  $x\in h^{-1}(x)$ all contain $\mathrm{Lie}(H_y)^\perp$.
By Borel fixed point Theorem (\cite{Sp}, Thm. 6.2.6) $\mathrm{Lie}(H_y)$ 
is the stabilizer of some $x \in X_y$, and the statement follows from the fact that
the codimension of this smallest subspace is exactly $\delta(y)$.
\end{proof}

From this proposition and Corollary \ref{weightcriterion} one gets a fairly workable 
criterion for proving support theorems in the line of those in \cite{N, CL1,CL2},
see the forthcoming \cite{MSV} where compactified Jacobian fibrations 
associated with family of planar curves are considered in detail.


\end{document}